\documentclass[12pt,a4paper,reqno]{amsart}

\usepackage[margin=1in,footskip=0.25in]{geometry}
\usepackage{amsmath, amsthm, amssymb}
\usepackage{graphicx}
\usepackage{mathrsfs}

\usepackage{enumitem}

\usepackage{float}

\usepackage{booktabs}

\usepackage{hyperref}
\usepackage{xcolor}

\newtheorem{theorem}{Theorem}

\newtheorem{proposition}[theorem]{Proposition}

\date{}

\allowdisplaybreaks
\begin{document}

\title[Zero-Hopf Bifurcation of Qi hyperchaotic Systems]
{Zero-Hopf Bifurcation of Qi hyperchaotic Systems}

\author{Sonia Isabel Renteria Alva$^1$}
\address{$^1$ Instituto de Matemática Pura e Aplicada, Estrada Dona Castorina 110, Jardim Botânico, Rio de Janeiro, 22460-320, Brazil}
\email{sonia.alva@impa.br}

\author{Pedro Iván Suárez Navarro$^2$}
\address{$^2$ Instituto de Matemática Pura e Aplicada,  Estrada Dona Castorina 110, Jardim Botânico, Rio de Janeiro, 22460-320, Brazil}
\email{ivan.suarez@impa.br}



\begin{abstract}
In this paper, we show a  zero–Hopf bifurcations in a four-dimensional hyperchaotic Qi system. While the hyperchaotic dynamics of this model have been extensively investigated, the existence and bifurcation of zero–Hopf equilibria have not been previously analyzed. We first characterize the equilibrium points and determine the parameter conditions under which the origin becomes a zero–Hopf equilibrium. Using second-order averaging theory, we establish the existence of up to four small amplitude periodic solutions bifurcating from the origin under suitable parameter perturbations. 
We also provide analytical expressions for the initial conditions of these solutions.
\end{abstract}

\maketitle

\section{Introduction and statement of the main results}

The study of chaotic and hyperchaotic dynamical systems has attracted considerable attention due to their rich dynamical behavior and numerous applications in physics, engineering, communications, and information security. Among the many chaotic models proposed in the literature, the Qi system introduced in \cite{qi2005analysis} constitutes an important example of a quadratic autonomous differential system exhibiting complex nonlinear dynamics. The original three-dimensional Qi system is given by
\begin{equation}\label{Qi3D}
\begin{aligned}
\dot{x} &= a(y-x)+yz,\\
\dot{y} &= cx-y-xz,\\
\dot{z} &= xy-bz,
\end{aligned}
\end{equation}
where $x$, $y$, and $z$ are state variables, and $a$, $b$, and $c$ are positive real parameters. Since its introduction, system \eqref{Qi3D} has been extensively studied from both theoretical and applied perspectives. Nevertheless, certain dynamical features remain only partially understood and continue to attract attention in the recent literature; see, for instance, \cite{zhang2025complex}.

To enrich the dynamical behavior of system \eqref{Qi3D}, several higher-dimensional extensions have been proposed. Among them, the four-dimensional hyperchaotic Qi system introduced in \cite{qi2008new,qi2009new} has received particular attention. This model preserves the quadratic structure of the original system while exhibiting hyperchaotic attractors, coexistence phenomena, synchronization properties, and other complex dynamical features.

In this paper we will study the existence of periodic orbits bifurcating from the zero-Hopf equilibria of the following hyperchaotic Qi system
\begin{equation}\label{s1}
\begin{aligned}
\dot{x} &= a(y-x)+yz,\\
\dot{y} &= b(x+y)-xz,\\
\dot{z} &= -cz-ew+xy,\\
\dot{w} &= -dw+kz+xy,
\end{aligned}
\end{equation}
where $x,y,z,w$ are state variables and $a,b,c,d,e,k\in\mathbb{R}$ are parameters. 

The hyperchaotic model \eqref{s1} was proposed by Qi et al. in \cite{qi2008new,qi2009new}, who studied its basic dynamical properties in the regime of positive parameter values. For suitable choices of the parameters, for instance $a=50$, $15.42\leq b\leq 27$, $c=13$, $d=8$, $e=33$, and $k=30$, it exhibits hyperchaotic behavior. Furthermore, the system is equivariant under the involution
$
\kappa:(x,y,z,w)\mapsto(-x,-y,z,w),
$
and has divergence
$
\operatorname{div}f=-a+b-c-d.
$
Consequently, it is dissipative whenever
$
a-b+c+d>0.
$

Bifurcations associated with multiple eigenvalue degeneracies play a fundamental role in the qualitative theory of differential equations, as they organize the local dynamics and may generate periodic, quasiperiodic, and more complex invariant sets. Among them, zero-Hopf bifurcations have attracted considerable attention in recent years \cite{kuznetsov1998elements,guckenheimer2013nonlinear}. In a four-dimensional autonomous system, a zero--Hopf equilibrium is characterized by the spectrum ${0,0,\pm i\omega}$ with $\omega>0$ and is often referred to as a zero-zero-Hopf bifurcation. Such configurations provide natural mechanisms for the emergence of periodic solutions and other invariant structures. They have been investigated in several classes of chaotic differential systems, including generalized Genesio systems \cite{DiabGuiraoVera2021}, generalized Michelson systems \cite{LlibreMakhlouf2016}, and three-dimensional chaotic systems with stable equilibria \cite{LlibreMessiasReinol2020}. More recently, attention has shifted toward hyperchaotic models, whose higher-dimensional phase spaces give rise to richer bifurcation scenarios. Representative examples include hyperchaotic Lorenz-type systems \cite{LlibreCandido2018HyperLorenz}, Lorenz--Haken hyperchaotic systems \cite{SuarezRenteria2024}, and other hyperchaotic differential systems of dimensions four and five \cite{llibre2021zero,DiabBustosLopezMartinez2023,diab2025zero}. These studies demonstrate the effectiveness of averaging theory as a powerful tool for establishing the existence of periodic solutions bifurcating from zero-Hopf singularities.

Despite the growing literature on zero–Hopf bifurcations in chaotic and hyperchaotic systems, no results appear to be available for the hyperchaotic Qi system. The main objective of this paper is to fill this gap. We first provide a complete characterization of the parameter values for which system \eqref{s1} possesses zero–Hopf equilibria. Subsequently, by means of second-order averaging theory, we establish the existence of up to four small-amplitude periodic solutions bifurcating from such equilibria. To the best of our knowledge, this constitutes the first rigorous analysis of zero–Hopf bifurcations in the hyperchaotic Qi system. Our results reveal dynamical phenomena occurring outside the classical hyperchaotic parameter regime originally considered in the literature, thereby uncovering a previously unexplored aspect of the model.

The existence of zero-Hopf equilibria requires a complete description of the equilibria of system \eqref{s1}. The next result provides such a characterization.

\begin{proposition}\label{prr1}
Assume that
$a\neq0,
$ $d\neq e,
$
and define
$\Delta=a^2+6ab+b^2$,
and
$
\Gamma_{\pm}
=
\frac{b(cd+ek)\bigl(3a+b\pm\sqrt{\Delta}\bigr)}
     {a(d-e)}$. Then system \eqref{s1} always admits the trivial equilibrium
$E_0=(0,0,0,0)$.

Moreover, whenever $\Delta\ge0$, the system possesses the nontrivial equilibria
$$
E_-^\pm=
\left(
\pm\frac{\sqrt{\Gamma_-}}{\sqrt2},
\mp\frac{a+b+\sqrt{\Delta}}
        {2\sqrt2\,b}\sqrt{\Gamma_-},
\frac{-a+b-\sqrt{\Delta}}{2},
-\frac{(c+k)(a-b+\sqrt{\Delta})}
       {2(d-e)}
\right),
$$
and
$$
E_+^\pm=
\left(
\pm\frac{\sqrt{\Gamma_+}}{\sqrt2},
\mp\frac{a+b-\sqrt{\Delta}}
        {2\sqrt2\,b}\sqrt{\Gamma_+},
\frac{-a+b+\sqrt{\Delta}}{2},
\frac{(c+k)(-a+b+\sqrt{\Delta})}
     {2(d-e)}
\right).
$$

Consequently, the system admits at most five real equilibrium points,
namely
$E_0$,$E_-^\pm$,$E_+^\pm$. 
The pair $E_-^\pm$ exists whenever $\Gamma_-\ge0$, whereas the pair
$E_+^\pm$ exists whenever $\Gamma_+\ge0$. Therefore, depending on the
signs of $\Delta$, $\Gamma_-$, and $\Gamma_+$, system \eqref{s1} may
possess one, two, three, four, or five real equilibrium points.
\end{proposition}

Proposition~\ref{prr1} follows from straightforward computations. Since system \eqref{s1} is symmetric, its nontrivial equilibria arise in symmetric pairs. We now turn to the characterization of zero-Hopf equilibria and begin with the trivial equilibrium point. The following result determines all parameter values for which the origin is a zero-Hopf equilibrium of system \eqref{s1}.

\begin{proposition}\label{prop}
The equilibrium point $E_0=(0,0,0,0)$ is a zero-Hopf equilibrium of
system \eqref{s1} if and only if
\[
a=b=0,
\qquad
d=-c,
\qquad
\omega^2=ek-c^2>0.
\]
\end{proposition}

In this work, we restrict our attention to the zero-Hopf bifurcation at the origin. Using averaging theory, we prove the existence of periodic solutions
bifurcating from it, as stated in the following theorem.

\begin{theorem}\label{teor2}
Let $
(a, b, d) = ( \varepsilon a_{1} + \varepsilon^2 a_{2} , \, \varepsilon b_{1} + \varepsilon^2 b_{2}, \, - c + \varepsilon d_{1}+\varepsilon^{2}d_{2})$,
where $a_1$, $a_2$, $b_1$, $b_2$, $d_1$, $d_2 \in \mathbb{R}$. The system \eqref{s1} has a zero-Hopf bifurcation at the equilibrium point $E_0$. More precisely, for  $\varepsilon \neq 0$  sufficiently small parameter, if $a_2\neq0$, $
b_2\neq0$, $
c+ e\neq 0
$, $\Delta_2=a_2^2+6a_2b_2+b_2^2>0$,  $ a_2b_2(c+e)<0$, and $( a_2^2+b_2\bigl(b_2\mp \sqrt{\Delta_2}\bigr) +a_2\bigl(6b_2\pm\sqrt{\Delta_2}\bigr) ) ( a_2^3+b_2^2\bigl(-b_2\pm \sqrt{\Delta_2}\bigr) +a_2^2\bigl(3b_2\pm\sqrt{\Delta_2}\bigr) +a_2b_2(-3b_2+2d_2) )\neq 0$, then   four periodic orbits  $\varphi_k(t, \varepsilon)$, $k=1,2,3,4$  bifurcates from $E_0$. These
periodic orbits satisfy
\begin{equation}
\scriptsize
\begin{aligned}
\varphi_1(0,\varepsilon)&= \varepsilon \left(
\frac{a_2+b_2-\sqrt{\Delta_2}}{2a_2}
\sqrt{\frac{a_2(a_2+3b_2+\sqrt{\Delta_2})(c^2-ek)}
{2b_2(c+e)}},
\sqrt{\frac{a_2(a_2+3b_2+\sqrt{\Delta_2})(c^2-ek)}
{2b_2(c+e)}},0,0 \right)
+\mathcal{O}(\varepsilon^2),\\[1ex]
\varphi_2(0,\varepsilon)&=
\varepsilon \left(
-\frac{a_2+b_2-\sqrt{\Delta_2}}{2a_2}
\sqrt{\frac{a_2(a_2+3b_2+\sqrt{\Delta_2})(c^2-ek)}
{2b_2(c+e)}},
-\sqrt{\frac{a_2(a_2+3b_2+\sqrt{\Delta_2})(c^2-ek)}
{2b_2(c+e)}},0,0 \right)
+\mathcal{O}(\varepsilon^2),\\[1ex]
\varphi_3(0,\varepsilon)&=
\varepsilon \left(
\frac{a_2+b_2+\sqrt{\Delta_2}}{2a_2}
\sqrt{\frac{a_2(a_2+3b_2-\sqrt{\Delta_2})(c^2-ek)}
{2b_2(c+e)}},
\sqrt{\frac{a_2(a_2+3b_2-\sqrt{\Delta_2})(c^2-ek)}
{2b_2(c+e)}},0,0 \right)
+\mathcal{O}(\varepsilon^2),\\[1ex]
\varphi_4(0,\varepsilon)&=
\varepsilon \left(
-\frac{a_2+b_2+\sqrt{\Delta_2}}{2a_2}
\sqrt{\frac{a_2(a_2+3b_2-\sqrt{\Delta_2})(c^2-ek)}
{2b_2(c+e)}},
-\sqrt{\frac{a_2(a_2+3b_2-\sqrt{\Delta_2})(c^2-ek)}
{2b_2(c+e)}},0,0 \right)
+\mathcal{O}(\varepsilon^2).
\end{aligned}
\normalsize
\end{equation}

Moreover, $\varphi_k(t, \varepsilon)$, $k=1,2,3,4$ can be locally stable or unstable depending on the chose of the
parameters $a_2$, $b_2$, $c$, $e$ and $k$.
\end{theorem}

The remainder of the paper is organized as follows. In Section~\ref{sec2:averg}  we present the results on the averaging theory that we need for proving Theorem~\ref{teor2}. Section~\ref{sec:proofs} contains the proofs of Proposition~\ref{prop}, which characterizes the zero-Hopf equilibria, and Theorem~\ref{teor2}, which establishes the existence of periodic solutions bifurcating from the origin by means of second-order averaging theory.
\section{The averaging theory of first and second order} \label{sec2:averg}
In this section, we recall the averaging theory of first and second  order for studying the periodic orbits of nonlinear differential systems that will be used in the proof of Theorem~\ref{teor2}. Their proofs and further discussions may be found in \cite{buicua2004averaging, Verhulst1996, candido2017persistence}.

\begin{theorem}\label{ThmAverging}
Let $A$ be an open subset of $\mathbb{R}^{n}$ and $\varepsilon_0>0$.
Consider the differential system
\begin{equation}\label{eq:sys}
\dot{x}(t) = \varepsilon F_1(x,t) + \varepsilon^2 F_2(x,t) + \varepsilon^3 R(x,t,\varepsilon),
\end{equation}
where
$$ F_1,F_2:A\times\mathbb{R}\to\mathbb{R}^{n}, \qquad
R:A\times\mathbb{R}\times(-\varepsilon_0,\varepsilon_0) \to \mathbb{R}^{n} $$
are continuous functions, $T$-periodic in the second variable, $F_1(\cdot,t)\in C^1(A)$ for all $t\in\mathbb{R}$, $F_1$, $F_2$, $R$, and $D_xF_1$ are locally Lipschitz with respect to $x$, and $R$ is differentiable with respect to $\varepsilon$.

Define $f,g:A\to\mathbb{R}^{n}$ by
\begin{equation}\label{eq:f}
f(z) =  \frac{1}{T} \displaystyle \int_0^T F_1(z,s)\,ds, 
\end{equation}
and
\begin{equation}\label{eq:g}
g(z) = \frac{1}{T} \int_0^T \left[ D_zF_1(z,s)\cdot \displaystyle \int_0^s F_1(z,t)\,dt + F_2(z,s)
\right] ds.
\end{equation}

\begin{enumerate}
\item[(a)] Set $h=f$ if $f\neq0$, otherwise set $h=g$. If there exists $p\in A$ such that
$$ h(p)=0 \qquad\text{and}\qquad \det Dh(p)\neq0, $$
then, for $|\varepsilon|>0$ sufficiently small, there exists a $T$-periodic solution
$\varphi(\cdot;\varepsilon)$ of system \eqref{eq:sys} such that
$$ \varphi(t,\varepsilon)-p=O(\varepsilon). $$

\item[(b)] For $p$ and $\varphi$ as in (a), if the eigenvalues of $Dh(p)$ have negative real parts, then the periodic solution $\varphi$ is stable.
If one of the eigenvalues of $Dh(p)$ has positive real part, then $\varphi$ is unstable.
\end{enumerate}

If $f\neq0$, then statement (a) provides the first-order averaging; otherwise, it provides the second-order averaging.
\end{theorem}
\section{Proof of results}\label{sec:proofs}

In this section, we will provide the proofs of Proposition \eqref{prop} and Theorem \eqref{teor2}.

\begin{proof}[Proof of proposition \ref{prop}] 

The Jacobian matrix of system \eqref{s1} evaluated at
$E_0=(0,0,0,0)$ is
\[
J=
\begin{pmatrix}
-a & a & 0 & 0\\
b & b & 0 & 0\\
0 & 0 & -c & -e\\
0 & 0 & k & -d
\end{pmatrix}.
\]

The characteristic polynomial  for the matrix $J$ is
\begin{equation*}\label{polzero}
\begin{aligned}
P(\lambda)
=&\,\lambda^4
+(a-b+c+d)\lambda^3  +\bigl(cd-b(c+d)+a(c+d-2b)+ek\bigr)\lambda^2\\
&+\bigl(a(cd-2b(c+d)+ek)-b(cd+ek)\bigr)\lambda -2ab(cd+ek).
\end{aligned}
\end{equation*}

In order that the equilibrium point $E_0$ be a zero-Hopf equilibrium the polynomial $P(\lambda)$ must be of the form
$\lambda^2(\lambda^2+\omega^2)$, $
\omega>0$. Then we obtain
$$
a=b=0,
\qquad
d=-c,
\qquad
ek-c^2>0,
\qquad
\omega=\sqrt{ek-c^2}.
$$

\end{proof}

\begin{proof}[Proof of Theorem \ref{teor2}]
Let  $(a, b, d) = ( \varepsilon a_{1} + \varepsilon^2 a_{2} , \, \varepsilon b_{1} + \varepsilon^2 b_{2}, \, - c + \varepsilon d_{1}+\varepsilon^{2}d_{2} )$,  where  $\varepsilon>0$ is a sufficiently  small parameter and $a_1$, $a_2$, $b_1$, $b_2$, $d_1$ and $d_2$ are real nonzero numbers. With these subsitutions, the system \eqref{s1} becomes
\begin{equation}\label{eq:sistema}
\begin{aligned}
\dot{x} &= y z-\varepsilon a_{1}(x-y)-\varepsilon^{2}a_{2}(x-y),\\
\dot{y} &= -xz+\varepsilon b_{1}(x+y)+\varepsilon^{2}b_{2}(x+y),\\
\dot{z} &= xy-cz-ew,\\
\dot{w} &= xy+kz+cw-\varepsilon d_{1}w-\varepsilon^{2}d_{2}w.
\end{aligned}
\end{equation}

Next, we rescale the variables. Let $(x,y,z,w)=( \epsilon X, \epsilon Y,  \epsilon Z,\epsilon W)$, then system \eqref{eq:sistema} is written as follows: 
\begin{equation}\label{sist-rescal}
\begin{aligned}
\dot{X} &= \varepsilon YZ +\varepsilon a_{1}(Y-X) +\varepsilon^{2}a_{2}(Y-X),\\[1ex]
\dot{Y} &= -\varepsilon XZ +\varepsilon b_{1}(X+Y) + \varepsilon^{2}b_{2}(X+Y),\\[1ex]
\dot{Z} &= -(cZ+eW) +\varepsilon XY,\\[1ex]
\dot{W} &= (kZ+cW) +\varepsilon\left(XY-d_{1}W\right) -\varepsilon^{2}d_{2}W.
\end{aligned}
\end{equation}

Now we shall write the linear part at the origin of the system \eqref{sist-rescal} when $\varepsilon = 0 $ into its real Jordan normal form, i.e., as

$$\left( \begin{array}{cccc}
0 & - \sqrt{ek - c^2} & 0 & 0 \\[2mm]
\sqrt{ek -c^2} & 0 & 0 & 0 \\[2mm]
0 & 0 & 0 & 0 \\[2mm]
0 & 0 & 0 & 0
\end{array}\right).$$
For doing that we consider the linear change
\begin{equation*}\label{eq-proof-4}
   \Big( x, y, z, w \Big) = \Bigg( W, Z, -\frac{cX+\sqrt{ek-c^{2}}\,Y}{k}, X \Bigg),
\end{equation*}
makes what we need. In these new variables and by denoting again the variables $(X, Y, Z,W) \to (x, y, z, w)$, the system \eqref{sist-rescal} is written as follows
\begin{equation}\label{sist-novo}
\begin{aligned}
\dot{x} &= -\sqrt{ek-c^{2}}\,y +\varepsilon\left(-d_{1}x+wz\right) -d_{2}x\,\varepsilon^{2}, \\[2mm]
\dot{y} &= \sqrt{ek-c^{2}}\,x +\varepsilon\frac{c d_{1}x-(c+k)wz}{\sqrt{ek-c^{2}}}
+\varepsilon^{2}\frac{c d_{2}x}{\sqrt{ek-c^{2}}}, \\[2mm]
\dot{z} &= \varepsilon\left( \frac{w\left(cx+\sqrt{ek-c^{2}}\,y\right)}{k} +b_{1}(w+z)
\right) +\varepsilon^{2} b_{2}(w+z),\\[2mm]
\dot{w} &= \varepsilon\left( -\frac{\left(cx+\sqrt{ek-c^{2}}\,y\right)z}{k}
+a_{1}(-w+z) \right) +\varepsilon^{2} a_{2}(-w+z).
\end{aligned}
\end{equation}
Now, we transform the differential system \eqref{sist-novo} to cylindrical coordinates $(r, \theta, z, w)$ by $x = r \cos\theta$, $y = r \sin\theta$, $z=z$, and $w=w,$ with $r>0$, obtaining 
\begin{equation}\label{syst:theta}
\begin{aligned}
\dot r =& -\frac{\varepsilon}{\sqrt{ek-c^{2}}}
\Bigg( \sqrt{ek-c^{2}}\,r(d_{1}+d_{2}\varepsilon)\cos^{2}\theta
+(c+k)wz\sin\theta \\
&\qquad\qquad -\cos\theta\Big( \sqrt{ek-c^{2}}\,wz + cr(d_{1}+d_{2}\varepsilon)\sin\theta \Big) \Bigg), \\[2mm]
\dot\theta =&\frac{1}{ 2\sqrt{ek-c^{2}}\,r } \Bigg( r \big( -2c^{2}+2ek+c\varepsilon(d_{1}+d_{2}\varepsilon) \big) - 2(c+k)wz\,\varepsilon\cos\theta \\
&\qquad\qquad + c r \varepsilon(d_{1}+d_{2}\varepsilon)\cos 2\theta + 2\sqrt{ek-c^{2}}\,\varepsilon \big( - wz+r(d_{1}+d_{2}\varepsilon)\cos\theta \big)\sin\theta \Bigg), \\[2mm]
\dot z =& \frac{\varepsilon}{k} \Big( k(w+z)(b_{1}+b_{2}\varepsilon) +c r w\cos\theta
+\sqrt{ek-c^{2}}\,r w\sin\theta \Big), \\[2mm]
\dot w =& -\frac{\varepsilon}{k} \Big( k(w-z)(a_{1}+a_{2}\varepsilon) +c r z\cos\theta +\sqrt{ek-c^{2}}\,r z\sin\theta \Big).
\end{aligned}
\end{equation}

By introducing $\theta$ as the new independent variable, we obtain 
\begin{equation}\label{syst:rzw}
\begin{aligned}
\frac{dr}{d\theta}  &=& \varepsilon F_{11}(r,z,w,\theta) +\varepsilon^{2} F_{12}(r,z,w,\theta) + \mathcal{O}(\varepsilon^{3}), \\[2mm]
\frac{dz}{d\theta} &=& \varepsilon F_{21}(r,z,w,\theta) + \varepsilon^2 F_{22}(r,z,w,\theta) + \mathcal O(\varepsilon^3),\\[2mm]
\frac{dw}{d\theta} &=& \varepsilon F_{31}(r,z,w,\theta) + \varepsilon^2 F_{32}(r,z,w,\theta) + \mathcal O(\varepsilon^3),
\end{aligned}
\end{equation}
where

\begin{equation*}
\begin{aligned}
F_{11} =& \frac{ d_{1}\sqrt{ek-c^{2}}\,r\cos^{2}\theta +(c+k)wz\sin\theta -\cos\theta\left(
\sqrt{ek-c^{2}}\,wz +cd_{1}r\sin\theta \right) }{c^{2}-ek}, \\
F_{12} =& \frac{1}{8( c^{2}-ek )^{2}r} \Bigg( 2\sqrt{ek-c^{2}} \left( c(d_{1}^{2}+2cd_{2})
-2d_{2}ek \right)r^{2} -4d_{1}(2c+k)\sqrt{ek-c^{2}}\,rwz \\
& \cos\theta +4\sqrt{ek-c^{2}} \left( \left( c(d_{1}^{2}+cd_{2}) -d_{2}ek \right)r^{2}
+2(c+k)w^{2}z^{2} \right)\cos 2\theta -8cd_{1}\sqrt{ek-c^{2}} \\
& rwz\cos 3\theta -4d_{1}k\sqrt{ek-c^{2}}\,rwz\cos 3\theta +2cd_{1}^{2}\sqrt{ek-c^{2}}\,r^{2}\cos 4\theta +4d_{1}(2c^{2}+ck \\
& -ek)\,rwz\sin\theta -2\left( \left( 2c^{2}(d_{1}^{2}+cd_{2}) -(d_{1}^{2}+2cd_{2})ek \right)r^{2} +2(2c^{2}+2ck-ek+k^{2})w^{2}z^{2} \right) \\ 
& \sin 2\theta +4d_{1}(2c^{2}+ck-ek)\,rwz\sin 3\theta +d_{1}^{2} (-2c^{2}+ek)r^{2}\sin 4\theta \Bigg),
\end{aligned}
\end{equation*}

\begin{equation*}
\begin{aligned}
F_{21} =& \frac{ b_{1}k(w+z) + crw\cos\theta + \sqrt{ek-c^{2}}\,rw\sin\theta }{k\sqrt{ek-c^{2}}}, \\
F_{22} =& -\frac{1}{4k(ek-c^{2})^{3/2}r} \Bigg( 2kr\Big( -(c+e)w^{2}z +b_{1}cd_{1}(w+z) +2b_{2}(c^{2}-ek)(w+z) \Big) \\
& +w\Big( d_{1}(2c^{2}+ek)r^{2} -4b_{1}k(c+k)wz -4b_{1}k(c+k)z^{2} \Big)\cos\theta +2r\Big( (-2c^{2}-ck+ek) \\
& w^{2}z +b_{1}cd_{1}k(w+z) \Big)\cos 2\theta +d_{1}(2c^{2}-ek)r^{2}w\cos 3\theta +2cd_{1}\sqrt{ek-c^{2}}\,r^{2}w\sin\theta \\
& -4b_{1}k\sqrt{ek-c^{2}}\,w^{2}z\sin\theta -4b_{1}k\sqrt{ek-c^{2}}\,wz^{2}\sin\theta +2b_{1}d_{1}k\sqrt{ek-c^{2}}\,rw\sin 2\theta \\
& +2b_{1}d_{1}k\sqrt{ek-c^{2}}\,rz\sin 2\theta -4c\sqrt{ek-c^{2}}\,rw^{2}z\sin 2\theta -2k\sqrt{ek-c^{2}}\,rw^{2}z\sin 2\theta \\
& +2cd_{1}\sqrt{ek-c^{2}}\,r^{2}w\sin 3\theta \Bigg),
\end{aligned}
\end{equation*}
\begin{equation*}
\begin{aligned}
F_{31} =& -\frac{ a_{1}k(w-z) + crz\cos\theta + \sqrt{ek-c^{2}}\,rz\sin\theta }{k\sqrt{ek-c^{2}}}, \\
F_{32} =& -\frac{1}{4k(ek-c^{2})^{3/2}r} \Bigg( 2kr\Big( -2a_{2}(c^{2}-ek)(w-z) +(c+e)wz^{2} +a_{1}cd_{1}(-w+z) \Big) \\
& -z\Big( 2c^{2}d_{1}r^{2} +4a_{1}ckw(-w+z) +k\big(d_{1}er^{2}+4a_{1}kw(-w+z)\big) \Big)\cos\theta +2r\Big( (2c^{2} \\
& +ck-ek)wz^{2} +a_{1}cd_{1}k(-w+z) \Big)\cos 2\theta +d_{1}(-2c^{2}+ek)r^{2}z\cos 3\theta -2cd_{1} \\
& \sqrt{ek-c^{2}}\,r^{2}z\sin\theta +4a_{1}k\sqrt{ek-c^{2}}\,w^{2}z\sin\theta -4a_{1}k\sqrt{ek-c^{2}}\,wz^{2}\sin\theta -2a_{1}d_{1}k \\
& \sqrt{ek-c^{2}}\,rw\sin 2\theta +2a_{1}d_{1}k\sqrt{ek-c^{2}}\,rz\sin 2\theta +4c\sqrt{ek-c^{2}}\,rwz^{2}\sin 2\theta +2k \\
& \sqrt{ek-c^{2}}\,rwz^{2}\sin 2\theta -2cd_{1}\sqrt{ek-c^{2}}\,r^{2}z\sin 3\theta \Bigg).
\end{aligned}
\end{equation*}

In  the notations of  Theorem~\ref{ThmAverging}, by taking $t=\theta$, $T=2\pi$, and $\textbf{x}=(r,z,w)$, we have 
\begin{equation*}
    \begin{aligned}
        F_1(r,z,w,\theta)=& (F_{11}( \textbf{x},\theta),
 F_{21}( \textbf{x},\theta), F_{31}( \textbf{x},\theta)),\\
 F_2(r,z,w,\theta)=& (F_{12}( \textbf{x},\theta),
 F_{22}( \textbf{x},\theta), F_{32}( \textbf{x},\theta)),
    \end{aligned}
\end{equation*}

so we are under the assumptions of this theorem. We calculate $f(r,z,w)$ obtaining
\begin{equation*}
\begin{aligned}
f(r,z,w)
&=\frac{1}{2\pi}\int_0^{2\pi}
F_1(r,z,w,\theta)\,d\theta  \\
&=
\frac{1}{2\pi\sqrt{ek-c^2}}
\left(
-\frac{d_1 r}{2},
\, b_1(w+z),
\, a_1(z-w)
\right).
\end{aligned}
\end{equation*}

 We look for the zeros $(r^{*},z^{*}, w^{*})$ of $f(r,z,w)$ with $r>0$. Since the unique zero of $f(r,z,w)$ is $(0,0,0)$ (or a continuum of zeros whenever $d_1=0$, $b_1$ or $a_1=0$), the first-order averaged function provides no information about the periodic solutions of system~\eqref{syst:rzw}.
In particular we cannot apply the averaging of first order. We pass then to the averaging of second order, assuming $f \equiv 0$. This makes $a_1 = b_1 = d_1 = 0$ and simplifies $F_2$ to $F_2(r, z, w, \theta) = \Big( \widetilde{F_{21}}, \widetilde{F_{22}}, \widetilde{F_{23}} \Big)$, with
\begin{equation*}
    \begin{aligned}
\widetilde{F_{21}} =& -\dfrac{1}{2\left(c^{2}-ek\right)^{2}r} \Bigg( d_{2}\left(ek-c^{2}\right)^{3/2}r^{2} -\sqrt{ek-c^{2}} \left( d_{2}\left(c^{2}-ek\right)r^{2}
+2(c+k)w^{2}z^{2} \right) \\
& \cos 2\theta +\left( cd_{2}\left(c^{2}-ek\right)r^{2}
+\left(2c^{2}+2ck-ek+k^{2}\right)w^{2}z^{2} \right)\sin 2\theta \Bigg), \\[2mm]  
\widetilde{F_{22}} =& - \dfrac{1}{2k\left(ek-c^{2}\right)^{3/2}} \Bigg( -\left(c+e\right)kw^{2}z +2b_{2}k\left(c^{2}-ek\right)(w+z) -w^{2}z \\
& \left( \left(2c^{2}+ck-ek\right)\cos 2\theta +\left(2c+k\right)\sqrt{ek-c^{2}}\sin 2\theta \right) \Bigg), \\[2mm] 
\widetilde{F_{23}} =& - \dfrac{1}{2k\left(ek-c^{2}\right)^{3/2}} \Bigg( k\left(
-2a_{2}\left(c^{2}-ek\right)(w-z) +\left(c+e\right)wz^{2} \right) +wz^{2} \\
& \left( \left(2c^{2}+ck-ek\right)\cos 2\theta +\left(2c+k\right)\sqrt{ek- c^{2}}\sin 2\theta \right) \Bigg).
    \end{aligned}
\end{equation*}
We now compute $g(r,z,w)$, which is required to apply Theorem~\ref{ThmAverging} at second order:
\begin{equation*}
 \begin{aligned}
g(r,z,w) =& \frac{1}{2\pi} \int_{0}^{2\pi} \left( D_{(r,z,w)} F_{1}(r,z,w,\theta) \int_{0}^{\theta} F_{1}(r,z,w,s)\,ds + F_{2}(r,z,w,\theta) \right) d\theta \\[2mm]
=& \Big( g_1(r,z,w), g_2(r,z,w), g_3(r,z,w) \Big),
 \end{aligned}
\end{equation*}
with
\begin{equation*}
    \begin{aligned}
g_1(r,z,w)=& \dfrac{r\left(c^{2}d_{2}-e\left(d_{2}k+w^{2}-z^{2}\right) +c\left(- w^{2}+z^{2}\right)\right)}{2\left(ek-c^{2}\right)^{3/2}}, \\[3mm]
g_2(r,z,w)=&-\dfrac{ -(c+e)w^{2}z+b_{2}(c^{2}-ek)(w+z) }{\left(ek-c^{2}\right)^{3/2}},\\[3mm]
g_3(r,z,w)=&-\dfrac{ -a_{2}(c^{2}-ek)(w-z)+(c+e)wz^{2} }{\left(ek-c^{2}\right)^{3/2}},
    \end{aligned}
\end{equation*}
and solving the nonlinear system given by $g(r,z,w) = 0$ we can conclude that the system has the next four solutions. 
\begin{equation*}
\begin{aligned}
S_0 = (0,0,0), \quad 
S_{1,2} = \left( 0,\, \pm z_0, \, \pm w_0 \right), \quad
S_{3,4} = \left( 0,\, \pm \widehat{z}_0, \, \pm \widehat{w}_0 \right),
\end{aligned}
\end{equation*}
where $z_0=\sqrt{\frac{a_2(a_2+3b_2+\sqrt{\Delta_2})(c^2-ek)}{2b_2(c+e)}}$, $w_0=\frac{a_2+b_2-\sqrt{\Delta}}{2a_2} z_0$, $\widehat{z}_0=\sqrt{\frac{a_2(a_2+3b_2-\sqrt{\Delta_2})(c^2-ek)}{2b_2(c+e)}}$, $\widehat{w}_0=\frac{a_2+b_2+\sqrt{\Delta_2}}{2a_2} \widehat{z}_0$, and $\Delta_2=a_2^2+6a_2b_2+b_2^2$. Note that $S_{1,2}$ and $S_{3,4}$ exists only if conditions $\Delta_2>0$ and  $ a_2b_2(c+e)<0$  are satisfied.

The solution $s_0$ corresponds to the equilibrium point and consequently it does not provide any periodic orbit. 
But the other four solutions provide four isolated periodic orbits $
(r_k(\theta,\varepsilon),
 z_k(\theta,\varepsilon),
 w_k(\theta,\varepsilon)),
$ for $ k=1,2,3,4$, i.e.,  by Theorem \ref{ThmAverging} there exist four limit cycles, for the differential system \eqref{syst:rzw} because,   for the solution $S_{1,2}$ the Jacobian matrix of $g$ evaluated at this point is given by
$$
Dg(S_{1,2})
=
\begin{pmatrix}
\frac{
- a_2^3
+ b_2^2(b_2 - \sqrt{\Delta_2})
- a_2^2(3 b_2 + \sqrt{\Delta_2})
+ a_2 b_2 (3 b_2 - 2 d_2)
}{
4 a_2 b_2 \sqrt{ ek-c^2}
}
& 0 & 0 \\[10pt]

0 &
-\frac{
b_2(a_2 + b_2 - \sqrt{\Delta_2})
}{
2 a_2 \sqrt{ ek-c^2}
}
&
\frac{
a_2 + \sqrt{\Delta_2}
}{
\sqrt{ek-c^2}
}
\\[10pt]

0 &
\frac{
b_2 - \sqrt{\Delta_2}
}{
\sqrt{ek-c^2}
}
&
\frac{
a_2(a_2 + b_2 + \sqrt{\Delta_2})
}{
2 b_2 \sqrt{ek-c^2}
}
\end{pmatrix}
$$

and its determinant is

    \begin{equation*}
        \begin{aligned}
                    \det \Big( \frac{\partial g}{\partial \mathbf{x}} (S_{1,2}) \Big) 
            =& -\frac{ 1 }{ 4a_2b_2(ek-c^2)^{3/2} } \left( a_2^2+b_2\bigl(b_2-\sqrt{\Delta_2}\bigr) +a_2\bigl(6b_2+\sqrt{\Delta_2}\bigr) \right) \\
            & \quad \left( a_2^3+b_2^2\bigl(-b_2+ \sqrt{\Delta_2}\bigr) +a_2^2\bigl(3b_2+\sqrt{\Delta_2}\bigr) +a_2b_2(-3b_2+2d_2) \right) \neq 0,
        \end{aligned}
    \end{equation*}

by assumptions.

Analogously, for the equilibria \(S_3\) and \(S_4\) the structure of the Jacobian matrices at \(S_{3,4}\)  is similar to that of \(S_{1,2}\) , differing only by sign changes in several terms involving $(\sqrt{\Delta_2})$. As a consequence, the characteristic polynomials associated with
\(\left( \frac{\partial g}{\partial \mathbf{x}} \right)(S_{3,4})\)
have the same algebraic structure as those previously obtained for \(S_{1,2}\), with analogous modifications in the coefficients. In particular, the determinant of the Jacobian at \(S_{3,4}\) is given by
\begin{equation*}
\begin{aligned}
\det \Big( \frac{\partial g}{\partial \mathbf{x}} (S_{3,4}) \Big)
=& -\frac{1}{4a_2 b_2 (ek - c^2)^{3/2}}
\Bigl(
a_2^2 - a_2(-6b_2 + \sqrt{\Delta_2}) + b_2(b_2 + \sqrt{\Delta_2})
\Bigr) \\
& \quad 
\Bigl(
a_2^3 - a_2^2(-3b_2 + \sqrt{\Delta_2})
- b_2^2(b_2 + \sqrt{\Delta_2})
+ a_2 b_2(-3b_2 + 2d_2)
\Bigr)
\neq 0.
\end{aligned}
\end{equation*}

Moreover, the characteristic polynomials of $
\left( \frac{\partial g}{\partial \mathbf{x}} \right)(S_{1,2})
$  coincide and are
given by
\begin{equation*}\label{polySolu}
    p(\lambda)=\lambda^3 + m_1\lambda^2 + m_2\lambda + m_3,
\end{equation*}
where the coefficients \(m_1\), \(m_2\), and \(m_3\) are given by
\begin{equation*}
\begin{aligned}
m_1 =& -\frac{1}{4a_2b_2\sqrt{ek-c^2}}
\Big( (a_2^2+b_2^2)(\sqrt{\Delta_2}-b_2) + a_2^3 + a_2b_2(b_2-2d_2) \Big), \\[6pt]
m_2 =& \frac{1}{4a_2^2b_2^2(c^2-ek)}
\Big(
a_2^6 + b_2^5(b_2-\sqrt{\Delta_2}) + a_2^5(5b_2+\sqrt{\Delta_2})
- a_2^2b_2^3(3b_2-5\sqrt{\Delta_2}+d_2) \\
& \quad + a_2^4b_2(-3b_2+2\sqrt{\Delta_2}+d_2)
+ a_2^3b_2(-22b_2^2+\sqrt{\Delta_2}d_2+b_2(-5\sqrt{\Delta_2}+d_2)) \\
& \quad + a_2b_2^3(5b_2^2+\sqrt{\Delta_2}d_2-b_2(2\sqrt{\Delta_2}+d_2))
\Big), \\[6pt]
m_3 =& \frac{1}{4a_2b_2(ek-c^2)^{3/2}}
\Big(
a_2^2 + b_2(b_2-\sqrt{\Delta_2}) + a_2(6b_2+\sqrt{\Delta_2})
\Big) \\
& \quad \Big(
a_2^3 + b_2^2(-b_2+\sqrt{\Delta_2})
+ a_2^2(3b_2+\sqrt{\Delta_2})
+ a_2b_2(-3b_2+2d_2)
\Big).
\end{aligned}
\end{equation*}

The parameters \(a_2, b_2, c_2, d_2, e,\) and \(k\) can be chosen such that the polynomial \(p(\lambda)\) has roots with real parts of arbitrary signs. Consequently, the periodic orbit bifurcating from \(S_1\) (or \(S_2\)) may be locally stable if all eigenvalues of \(p(\lambda)\) have negative real parts. According to the Routh--Hurwitz criterion (see ~\cite{gantmacher2005applications}, p.~231), this occurs under the conditions $m_1>0$, $m_2>0$, $D_2>0$, $m_3>0$, where
\begin{equation*}
\begin{aligned}
D_2 =& -\frac{1}{8a_2^3b_2^3(ek-c^2)^{3/2}}
\Big(
- a_2^9 + 3a_2^7b_2(b_2-\sqrt{\Delta_2}) + b_2^8(b_2-\sqrt{\Delta_2})
- 3a_2b_2^7(-2b_2+\sqrt{\Delta_2}) \\
& \quad - a_2^8(6b_2+\sqrt{\Delta_2})
+ a_2^6b_2^2(43b_2+8\sqrt{\Delta_2}+2d_2)
- _2^3b_2^4\bigl(43b_2^2-19b_2\sqrt{\Delta_2}+d_2(-2\sqrt{\Delta_2}  \\
& \quad +d_2)\bigr) + a_2^4b_2^2\bigl(-3b_2^3+2b_2^2(\sqrt{\Delta_2}-2d_2)
+\sqrt{\Delta_2}d_2^2+b_2d_2(-2\sqrt{\Delta_2}+d_2)\bigr) \\
& \quad + a_2^5b_2^2\bigl(3b_2^2+19b_2\sqrt{\Delta_2}+d_2(2\sqrt{\Delta_2}+d_2)\bigr) \\
& \quad + a_2^2b_2^4\bigl(-3b_2^3+\sqrt{\Delta_2}d_2^2-b_2d_2(2\sqrt{\Delta_2}+d_2)+2b_2^2(4\sqrt{\Delta_2}+d_2)\bigr)
\Big).
\end{aligned}
\end{equation*}

On the other hand, if at least one eigenvalue of \(p(\lambda)\) has positive real part, then the periodic orbit is locally unstable.

Analogously, the stability analysis for \(S_3\) and \(S_4\) follows the same Routh--Hurwitz framework as for \(S_1\) and \(S_2\). In particular, since the characteristic polynomials associated with \(S_3\) and \(S_4\) have the same algebraic structure (up to sign changes in the terms involving \(\sqrt{\Delta_2}\)), one can derive analogous conditions on the corresponding coefficients to determine the local stability or instability of the associated periodic orbits.

Again by Theorem~\ref{ThmAverging} these four limit cycles satisfy
\begin{align*}
(r_k(0,\varepsilon), z_k(0,\varepsilon),
 w_k(0,\varepsilon))
&=(0,z_k,w_k)+\mathcal{O}(\varepsilon).
\end{align*}
These four limit cycles in the differential system~\eqref{syst:theta} are the limit cycles
$$
(r_k(t,\varepsilon),\theta_k(t,\varepsilon),
 z_k(t,\varepsilon),
 w_k(t,\varepsilon)),\qquad k=1,2,3,4,
$$
such that
\begin{align*}
(r_k(0,\varepsilon),\theta_k(0,\varepsilon),
 z_k(0,\varepsilon), w_k(0,\varepsilon))
&=
\left(
0,\sqrt{ek-c^2}\,t,z_k,w_k\right)+\mathcal{O}(\varepsilon).
\end{align*}

And in the differential system~\eqref{sist-novo} are the limit cycles
\[
(X_k(t,\varepsilon),Y_k(t,\varepsilon),
Z_k(t,\varepsilon),
W_k(t,\varepsilon)),
\qquad k=1,2,3,4
\]
such that
\begin{align*}
(X_k(t,\varepsilon),Y_k(t,\varepsilon),
Z_k(t,\varepsilon), W_k(t,\varepsilon))
&=
(0,0,z_k, w_k)
+\mathcal{O}(\varepsilon).
\end{align*}

Finally, taking into account the change of variables~\eqref{sist-rescal} we obtain the two
limit cycles
$$
(x_k(t,\varepsilon),y_k(t,\varepsilon),z_k(t,\varepsilon),w_k(t,\varepsilon)),
\qquad k=1,2,3,4
$$
of the differential system~\eqref{eq:sistema} satisfying
\begin{align*}
(x_1(0,\varepsilon),y_1(0,\varepsilon),z_1(0,\varepsilon),w_1(0,\varepsilon))
&=(
\varepsilon w_k, \varepsilon z_k, 0, 0) +\mathcal{O}(\varepsilon^2).
\end{align*}
From these last expressions, it is clear that these four limit cycles bifurcate from the equilibrium point $E_0$ when $\varepsilon=0$.
This completes the proof of Theorem~\ref{teor2}.
\end{proof}


\section*{Acknowlegements}
The first author acknowledges partial support from the 2026 Postdoctoral Program at the Instituto de Matemática Pura e Aplicada (IMPA) and from a scholarship granted by the Fundação Arthur Bernardes (FUNARBE).
The second author acknowledges partial support from CNPq under grant No.~169201/2023-6.


%
\bibliographystyle{acm}
\bibliography{sample}

@article{zhang2025complex,
  title={Complex bifurcations and new types of structure uncovered in the Qi system},
  author={Zhang, Enrui and Li, Xianyi},
  journal={Mathematics and Computers in Simulation},
  year={2025},
  publisher={Elsevier}
}

@article{llibre2021zero,
  title={The zero-Hopf bifurcations of a four-dimensional hyperchaotic system},
  author={Llibre, Jaume and Tian, Yuzhou},
  journal={Journal of Mathematical Physics},
  volume={62},
  number={5},
  year={2021},
  publisher={AIP Publishing}
}

@article{diab2025zero,
  title={Zero-Hopf bifurcation of a 5D hyperchaotic quadratic polynomial differential systems},
  author={Diab, Zouhair and Guirao, Juan LG and Llibre, Jaume},
  journal={Mathematics and Computers in Simulation},
  volume={238},
  pages={383--387},
  year={2025},
  publisher={Elsevier}
}

@article{SuarezRenteria2024,
  author  = {Su{\'a}rez Navarro, Pedro Iv{\'a}n and Renteria Alva, Sonia Isabel},
  title   = {Four-dimensional zero-Hopf bifurcation for a Lorenz--Haken system},
  journal = {Matem{\'a}tica Contempor{\^a}nea},
  year    = {2024},
  volume  = {50},
  pages   = {1--18}
}

@article{DiabGuiraoVera2021,
  author  = {Diab, Zouhair and Guirao, Juan L. G. and Vera, Juan A.},
  title   = {Zero-Hopf bifurcation in a generalized Genesio differential equation},
  journal = {Mathematics},
  volume  = {9},
  number  = {4},
  pages   = {354},
  year    = {2021}
}

@article{LlibreMakhlouf2016,
  author  = {Llibre, Jaume and Makhlouf, Abdelghani},
  title   = {Zero-Hopf bifurcation in the generalized Michelson system},
  journal = {Chaos, Solitons \& Fractals},
  volume   = {89},
  pages    = {228--231},
  year     = {2016}
}

@article{LlibreMessiasReinol2020,
  author  = {Llibre, Jaume and Messias, Marcelo and de Carvalho Reinol, Andr{\'e}},
  title   = {Zero-Hopf bifurcations in three-dimensional chaotic systems with one stable equilibrium},
  journal = {International Journal of Bifurcation and Chaos},
  volume   = {30},
  number   = {13},
  pages    = {2050189},
  year     = {2020}
}

@article{DiabBustosLopezMartinez2023,
  author  = {Diab, Zouhair and de Bustos, Mar{\'i}a T. and L{\'o}pez, Mar{\'i}a A. and Mart{\'i}nez, Rafael},
  title   = {The zero-Hopf bifurcations of a new hyperchaotic system},
  journal = {Applied Mathematics and Nonlinear Sciences},
  volume   = {8},
  number   = {2},
  pages    = {2251--2260},
  year     = {2023}
}

@article{LlibreCandido2018HyperLorenz,
  author  = {Llibre, Jaume and C{\^a}ndido, M. R.},
  title   = {Zero-Hopf bifurcations in a hyperchaotic Lorenz system II},
  journal = {International Journal of Nonlinear Sciences},
  volume   = {25},
  pages    = {3--26},
  year     = {2018}
}

@book{kuznetsov1998elements,
  title={Elements of applied bifurcation theory},
  author={Kuznetsov, Yuri A},
  year={1998},
  publisher={Springer}
}

@book{guckenheimer2013nonlinear,
  title={Nonlinear oscillations, dynamical systems, and bifurcations of vector fields},
  author={Guckenheimer, John and Holmes, Philip},
  year={2013},
  publisher={Springer Science \& Business Media}
}

@article{qi2005analysis,
  title={Analysis of a new chaotic system},
  author={Qi, Guoyuan and Chen, Guanrong and Du, Shengzhi and Chen, Zengqiang and Yuan, Zhuzhi},
  journal={Physica A: Statistical Mechanics and its Applications},
  volume={352},
  number={2-4},
  pages={295--308},
  year={2005},
  publisher={Elsevier}
}

@article{qi2009new,
  title={A new hyperchaotic system and its circuit implementation},
  author={Qi, Guoyuan and van Wyk, Michael Antonie and van Wyk, Barend Jacobus and Chen, Guanrong},
  journal={Chaos, Solitons \& Fractals},
  volume={40},
  number={5},
  pages={2544--2549},
  year={2009},
  publisher={Elsevier}
}

@article{qi2008new,
  title={On a new hyperchaotic system},
  author={Qi, Guoyuan and van Wyk, Micha{\"e}l Antonie and van Wyk, Barend Jacobus and Chen, Guanrong},
  journal={Physics Letters A},
  volume={372},
  number={2},
  pages={124--136},
  year={2008},
  publisher={Elsevier}
}

@article{buicua2004averaging,
  title={Averaging methods for finding periodic orbits via Brouwer degree},
  author={Buic{\u{a}}, Adriana and Llibre, Jaume},
  journal={Bulletin des sciences mathematiques},
  volume={128},
  number={1},
  pages={7--22},
  year={2004},
  publisher={Elsevier}
}

@book{Verhulst1996,
  author    = {Ferdinand Verhulst},
  title     = {Nonlinear Differential Equations and Dynamical Systems},
  edition    = {2},
  series     = {Universitext},
  publisher  = {Springer-Verlag},
  address    = {Berlin},
  year       = {1996},
  note        = {Translated from the 1985 Dutch original},
}

@article{candido2017persistence,
  title={Persistence of periodic solutions for higher order perturbed differential systems via Lyapunov--Schmidt reduction},
  author={C{\^a}ndido, Murilo R and Llibre, Jaume and Novaes, Douglas D},
  journal={Nonlinearity},
  volume={30},
  number={9},
  pages={3560--3586},
  year={2017},
  publisher={IOP Publishing}
}

@book{gantmacher2005applications,
  title={Applications of the Theory of Matrices},
  author={Gantmacher, Feliks Rouminovich and Brenner, Joel Lee},
  year={2005},
  publisher={Courier Corporation}
}

\end{document}